\documentclass[11pt]{amsart}
\usepackage{amsfonts, amsmath, amssymb, mathrsfs,latexsym}
\usepackage[all]{xy}
\usepackage{graphicx}
\usepackage[colorlinks=true,linkcolor=blue,urlcolor=blue,
citecolor=blue]{hyperref}

\usepackage[labelfont=small]{subfig}

\usepackage[draft]{fixme}
\makeatletter
\@namedef{subjclassname@2010}{%
 \textup{2010} Mathematics Subject Classification}
\makeatother

\newtheorem{theorem}{Theorem}[section]
\newtheorem{lemma}[theorem]{Lemma}
\newtheorem{proposition}[theorem]{Proposition}
\newtheorem{corollary}[theorem]{Corollary}
\newtheorem{definition}[theorem]{Definition}

\newtheorem{remark}[theorem]{Remark}

\newcommand \ZZ {\mathbb{Z}}
\newcommand \CC {\mathbb{C}}
\newcommand \RR {\mathbb{R}}

\newcommand \im {\text{im}}

\newcommand \aaa {\bar{a}}

\newcommand \cc {\bar{c}}
\newcommand \dd {\bar{d}}
\newcommand \ee {\bar{e}}
\newcommand \ff {\bar{f}}
\newcommand \HH {\bar{D}}

\newcommand \tr {\mathbb{T}}
\newcommand \pr {\mathbb{P}}

\begin{document}
\title{Tropical mixtures of star tree metrics}
\author{Mar\'ia Ang\'elica Cueto}
\thanks{The author was
  supported by a UC Berkeley Chancellor's Fellowship and the Laboratory for Mathematical and Computational Biology at UC Berkeley.}
\email{macueto@math.columbia.edu} \address{Mathematics Department,
  Columbia University, New York, NY, 10027, USA}


\begin{abstract} We study tree metrics that can be realized as a
  mixture of two star tree metrics. We prove that the only trees
  admitting such a decomposition are the ones coming from a tree with
  at most one internal edge, and whose weight satisfies certain linear
  inequalities.  We also characterize the fibers of the
  corresponding mixture map.  In addition, we discuss the general
  framework of tropical secant varieties and we interpret our results
  within this setting. Finally, we show that the set of 
  tree metric ranks of metrics on $n$ taxa is unbounded.
  \end{abstract}
\maketitle

\section{Introduction}\label{s:Intro}

In the present paper we investigate tropical mixtures of two
\emph{star tree metrics}. Here, star tree metrics on $n$ taxa are star
trees with $n$ leaves labeled 1 through $n$, equipped with nonnegative
weights on all its edges. Tropical mixture of two metrics $D, \bar{D}$
on $n$ taxa are defined by point-wise maxima of
metrics on $n$ taxa, i.e.
\[(D\oplus \HH)_{(i,j)}=\max\{ D_{(i,j)}, \HH_{(i,j)}\} \text{ for all
} 1\leq i,j\leq n.\] This study was inspired by the recent work of
Matsen, Mossel and Steel on phylogenetic mixtures of trees
(\cite{MMS}, \cite{MatSt}), as well as by some open questions
regarding tropical secant varieties of linear spaces
(\cite{tropicalDevelin}, \cite{tropicalDraisma}).

Matsen and his collaborators consider phylogenetic mixtures (convex
combinations of site pattern frequencies) of two weighted trees of the
same topological type, and they characterize the conditions under
which these give rise to weighted trees of a different topology.  In
particular, they pay special attention to mixtures of star tree
metrics on four taxa, since tree topologies are characterized by the
subtrees spanned by four of its leaves. Their interest in studying
these mixtures lies in the \emph{branch repulsion phenomenon}. More
precisely, they observed that if we mix together two weighted trees on
four taxa with the same topology, with short internal branch lengths
and with pendant edges alternating being long and short, we are likely
to get a tree with a different topology. Intuitively, this phenomenon
implies that we can approximate any tree topology we want by mixing
together two star tree metrics. Thus, the interest in understanding
mixtures of star tree metrics.



In this paper, rather than being interested in phylogenetic mixtures, we
focus our attention on \emph{tropical mixtures} of star tree
metrics. 
Tropical mixtures arise naturally in the context of probability theory
as approximations of mixtures of multivariate Poisson distributions
(\cite[Chapter 3]{ASCB}) and logdet transforms defined for Markov
models on trees (\cite[Theorem~8.4.3]{Phylog},
\cite{ReconstructingTrees}).  Since log-limits of secant varieties of
toric varieties correspond to tropical secant varieties of linear
spaces, one should expect similar results between convex mixtures and
tropical mixtures of star tree metrics. 
Theorem~\ref{thm:mainResult} shows this connection.

In the language of algebraic geometry, the image of the convex mixture
map of two star tree metrics studied by Matsen et al. corresponds to
secant varieties of star tree metrics. If we allow negative weights
for the edges of our star trees, then mixtures of these objects
correspond to secant varieties of star trees.
%
%
Since the convex mixture map is polynomial, we can tropicalize both
the space of star trees and the map to obtain a piecewise linear map: the
\emph{tropical mixture map}. Its image will be the tropical first
secant of an $n$-dimensional linear subspace of $\RR^{\binom{n}{2}}$
corresponding to the tropicalization of the space of star trees
(\cite{tropicalDevelin}). If we replace the space of star trees by the
space of star tree metrics and we tropicalize both the space and the
mixture map, we obtain a piecewise linear map parameterizing the
first tropical secant of a polyhedral cone rather than of a linear
space.  
%
The extremal rays of this cone are the cut metrics $d_i$ assigning
$d_i(i,j)=1$ for all $j\neq i$ and where all other distances are
zero. These cut metrics also span the space of star trees. The
sign constraints on the coefficients of the linear combinations
in this cone reflect the sign constraints on the edge weights of our star trees.
While star trees admit edges with negative weights, this is forbidden
in the biological framework of star tree metrics. 

One interesting question that arises from these two settings involves
\emph{star tree} and \emph{star tree metric ranks} of a symmetric
matrix. Here, by star tree (resp.\ star tree metric) rank we mean the
minimum number $k$ of star trees (resp.\ star tree metrics) required to
express a given metric as a tropical mixture of $k$ star trees
(resp.\ star tree metrics). By relaxing the positivity condition on the
admissible edge weights, one
can show that every tree metric on $n$ taxa can be written as the
tropical mixture of $n-2$ star trees (\cite[Theorem 5]{DM}). However,
the analogous question for metrics and star tree metrics is more
delicate. Indeed, we prove that most \emph{cut metrics} on $n$ taxa
(defined by partitions of the $n$ taxa, assigning pairwise distance 0
to pairs of taxa on the same subset, and distance 1 otherwise) 
are not tropical mixtures of finitely many star tree metrics. This
implies that the set of star tree metric ranks is infinite. 


The paper is organized as follows. In Section 2 we recall the basic
definitions on distances over finite sets and tree metrics, and we
present our main result: a complete characterization of tree metrics
with star tree metric rank at most two, together with a description of
the fibers of the mixture map (Theorem~\ref{thm:mainResult}). 
All proofs are deferred to Section~\ref{s:proofs}.

In Section~\ref{s:TropicalSecantVarieties} we switch gears and we
describe the general framework of tropical secants of linear spaces as
treated in \cite{tropicalDevelin, tropicalDraisma}. In the case of star trees with
arbitrary weights, the associated linear space is spanned by $n$
vectors $\{r_i=\sum_{j<i} e_{ji}+\sum_{j>i}e_{ij}, 1\leq i \leq
n\}$. 
In the spirit of~\cite{tropicalDevelin, tropicalDraisma}, we can associate to this
linear space a point configuration in $\RR^{\binom{n}{2}}$, namely the
set of vertices of the well-studied second hypersimplex $\Delta(2,n)$. Tropical
mixtures  correspond to certain regular subdivisions of this
polytope. 
Tropical secants of star tree metrics have the same underlying
point configuration, but the crucial difference between these two
settings lies in the admissible regular subdivisions of
$\Delta(2,n)$, which we call \emph{positive regular
  subdivisions}. Moreover, we show they are different even in the
simplest case of trees on four taxa. In particular, the existence of
metrics with infinite star tree metric rank implies that regular
subdivisions of $\Delta(2,n)$ need not be positive. As a corollary, we
show that the well-known problem of characterizing metrics on $n$ taxa
with finite tree metric ranks cannot be approached by studying star
tree ranks (see \cite[Chapter 3]{ASCB} for a conjecture on this
topic). We finish this paper with several open questions on tropical
secant varieties of polyhedral cones, which provide the natural
setting to investigate mixtures of tree metrics.





\section{Basic definitions and main result}\label{s:BasicDefinition}
We begin by providing notation and basic definitions. We fix $[n]=\{1,
\ldots, n\}$.

\begin{definition}
  A \emph{tree} $T$ \emph{on $n$ taxa} is a connected
  graph with no cycles and with leaves labeled by $[n]$. We let
 $E(T)$ be the edges of $T$.
\end{definition}
It is well-known that trees are completely determined by
their list of quartets, i.e.~subtrees spanned by four
leaves. For example, the quartet $(ij|kl)$
represents the subtree in Figure~\ref{fig:quartet}.
\begin{figure}[ht]
  \centering
  \includegraphics[scale=0.35]{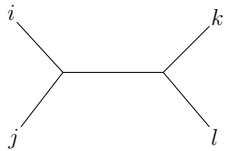}
\caption{Subtree determined by the quartet $(ij|kl)$.}
\label{fig:quartet}
\end{figure}
Many authors have discussed this tree representation and the minimum
number of quartets required to characterize the topological type of a
tree (\cite{314547}, \cite{quartetRepresentation}, \cite[$\S
5.4.2$]{treeOfLife}).

Given a tree $T$, we can assign nonnegative weights to
all edges in $T$ via a map $w\colon E(T) \to \RR_+$. The pair 
$(T,w)$ is called a \emph{weighted tree}. Throughout this section all
trees are weighted unrooted trees on $n$ taxa.

Given $i,j\in [n]$ and $(T,w)$ a tree, we define the \emph{distance}
between $i$ and $j$ to be $ d_T(i,j)=\sum_{e\in (i\rightarrow j)}w(e)$,
where $e$ varies along all edges in the unique path from $i$ to $j$.
We omit the subscript $T$ when understood. 

\begin{definition}\label{def:dissMap} A \emph{dissimilarity map} is a symmetric matrix
  $D\in \RR_+^{n\times n}$ such that $D_{ii}=0$ for all $i\in
  [n]$. Alternatively, whenever it is convenient for us we encode
  $D$ by a vector in $\RR_+^{\binom{n}{2}}$ corresponding the
  upper triangular portion of the original symmetric matrix.
  A dissimilarity map defines a \emph{metric}
  if it satisfies the triangular inequality $D_{ij}\leq D_{ik}
  +D_{kj}$ for all triples $i,j,k$.
\end{definition}
\begin{definition}
  A dissimilarity map $D$ is a \emph{tree metric} if there exists a
(unique)  weighted tree $(T,w)$ such that $D=d_T$. By abuse of notation we
  sometimes denote a tree metric by its unique associated weighted tree.
\end{definition}
  From the definition, it is straightforward to check that the set of
  dissimilarity maps is a closed pointed rational polyhedral cone
  isomorphic to $\RR_+^{\binom{n}{2}}$. Likewise, the set of
  metrics is a full dimensional closed pointed rational polyhedral
  cone contained in the cone of dissimilarity maps. However, 
 the space of tree metrics is a non-convex cone. It is a polyhedral
 complex, whose cells are convex cones.  Each maximal cell is a closed rational
polyhedral cone corresponding to a fixed tree topology and its
degenerations (obtained by setting some edge lengths to be zero). 
Its extremal rays are the compatible \emph{split metrics} $d_A$ defined by
subsets $A$ of $[n]$, that
determine the topology of the
tree $T$.  We define such split metrics as 
$d_A(i,j)=1$ if $\{i,j\}\not\subset A$ and $\{i,j\}\not\subset [n]\smallsetminus
A$, and $0$ in all other cases. The compatibility condition among $d_A$
and $d_{A'}$ says that one of the four intersections of $A$ or its
complement and $A'$ or its complement is empty. For a given tree
topology,  the cone spanned by the corresponding
compatible split metrics is isomorphic to the orthant
$\RR_+^{|E(T)|-n}$ (\cite[Prop~2.37]{ASCB}).
Tree metrics are
characterized algebraically by the ``Four point condition''
\cite[Theorem~2.36]{ASCB}:
\begin{theorem}\label{thm:4ptCondition} (Four point condition)
  A dissimilarity map $D\subset R_{+}^{\binom{n}{2}}$ is a tree metric \emph{if and only if} for
  any 4-tuple $i,j,k,l \in [n]$, the maximum among
  \[
\{D_{ij}+D_{kl}, D_{ik}+D_{jl}, D_{il}+D_{jk}\}
\]
is attained at least \emph{twice}. 
In particular, a tree metric $D$ is realized by a star
tree \emph{if and only if}
$D_{ij}+D_{kl}=D_{ik}+D_{jl}=D_{il}+D_{jk}$ for any $i,j,k,l$.
\end{theorem}
\noindent In Section~\ref{s:TropicalSecantVarieties} we will view this cone as the nonnegative points in the
tropicalization of a toric variety.
%
Next, we define mixtures of dissimilarity maps:
\begin{definition}
  Let $D, \HH$ be two dissimilarity maps. We define the
  \emph{tropical mixture} of $D$ and $\HH$ as $(D\oplus \HH)_{ij}=\max\{D_{ij},
  \HH_{ij}\}=D_{ij} \oplus \HH_{ij}$ for all $i,j\in [n]$. This
  generalizes to mixtures of any number of dissimilarity maps in the
  natural way. 

\end{definition}
To simplify notation, throughout the paper we refer to this construction
as the mixture of $D$ and $\HH$.
 From the definition it is straightforward to see that mixtures of
  metrics are metrics themselves. Since the Four point condition with
  non distinct indices provides the metric condition, we only need to
  check tuples of four \emph{distinct} indices.



  We define the \emph{mixture map} as the function $\phi\colon
  \text{\{tree metrics\}}^2 \to \text{\{metrics\}}$ sending $(D,\HH)
  \mapsto D\oplus \HH$. Since $ \lambda (D \oplus \HH)= (\lambda D)
  \oplus (\lambda \HH)$ for $\lambda \in \RR_+$, it follows that the
  image of $\phi$ is also a pointed cone inside
  $\RR_+^{\binom{n}{2}}$.  The commutativity of the mixture map can be
  interpreted by a natural action of the group $\ZZ_2$ on pairs of
  tree metrics, sending a pair $(D,\HH)$ to $(\HH,D)$. The fibers of
  the mixture map will be a polyhedral complex that is closed under
  this action, mapping cells to cells. The fiber of $\phi$ over a
  generic point would be pure, of dimension two with at least two
  distinct orbits.  We will see several examples of this behavior
  later on.

  We now focus our attention on the restriction of $\phi$ to
 star tree metrics:
\[
\phi\colon \text{\{star tree metrics\}} \times \text{\{star tree metrics\}} \to
\text{\{metrics\}}.
\]
The image of this restriction is also a pointed cone. The main goal of
this paper is to characterize the set $\im \,\phi \cap \text{\{tree
  metrics\}}$. Along the way, we determine the fibers of $\phi$. Since
$D=D\oplus D$, we already know that $\text{\{star tree metrics\}}
\subset \im\, \phi$. Thus, we need only study when a non-star tree
metric belongs to the image of $\phi$. We now present the desired
characterization. Our building blocks with be the trees on four and
five taxa.



\begin{lemma}\label{lm:n=4}
  Assume $n=4$ and let $T$ be the quartet $(ij|kl)$ with edge weights
 $e_i, e_j, e_k, e_l$ and $g$ as in Figure~\ref{sf:mainThma}. 
  Then, $T$ is a mixture of two star tree metrics \emph{if and only if}
  $e_i,e_j\geq g$ or $e_k,e_l\geq g$.
  Moreover for each pair of inequalities, the fiber over $T$ is a
  polyhedral complex in $\RR^8$ whose maximal cells are polyhedra of
  dimension at most two. Each maximal cell is affinely isomorphic to
  either a product of (degenerate) intervals, a (degenerate) trapezoid
  or a (degenerate) triangle in $\RR^2$, and there are either eight or
  sixteen of them. The fibers are closed under
  the action of $\ZZ_2$ on pairs of star tree metrics.  
\end{lemma}
%
%
\noindent Notice that the fibers over boundary points where some equality holds
consist of degenerate star trees, 
because some edges have weight zero. 

It is interesting to compare the result of the previous lemma with the
analogous one for \emph{convex
  mixtures} of two star trees on four taxa \cite[Prop~2.2]{MMS}. The tree
metrics that are convex mixture of two star tree metrics are exactly the
star trees and the quartet trees where the internal branch length is
shorter than the sum of the weights of \emph{any} pair of
non-adjacent edges.

\begin{lemma}\label{lm:n=5}
  Assume $n=5$ and let $T$ be a non-star weighted tree on five
  taxa. Denote by $e_t$ the weight of the edge attached to the
  leaf~$t\in [n]$.  Then, $T$ is a mixture of two star tree metrics
  \emph{if and only if} $T$ has \emph{only one internal edge} (labeled
  $g$), and $e_i,e_j\geq g$, where $i,j$ are the leaves of the unique
  cherry attached to one endpoint of the internal edge (see
  Figure~\ref{sf:mainThmb}).
  The fiber of $\phi$ over such a tree $T$ is a polyhedral complex of
  dimension at most two with eight maximal cells, that is closed under
  the natural symmetry action of $\ZZ_2$.
\end{lemma}
In particular, from the lemma we know that no trivalent tree metric is
a mixture of two star tree metrics.
The characterization for arbitrary number of taxa follows from the previous two
lemmas.
\begin{theorem}\label{thm:mainResult}
  Let $T$ be a non-star tree metric. Then $T$ is a
  mixture of two star tree metrics \emph{if and only if} one of the
  following conditions hold:
  \begin{enumerate}
  \item $n=4$, $T$ is the quartet $(ij|kl)$ and $e_i,e_j\geq g$ or
    $e_k,e_l\geq g$, where $g$ denotes the weight of the internal edge
    (see Figure~\ref{sf:mainThma});
  \item $n\geq 5$, $T$ has only one internal edge with weight $g$, and
    if $I$ and $J$ are a partition of $[n]$ into labels of leaves
    attached to the two extremal points of the internal edge of $T$,
    and we assume $|I|\leq |J|$, we have two cases:
  \begin{enumerate}
  \item $|I|=2$ and $e_{i_1}, e_{i_2}\geq g$ (see Figure~\ref{sf:mainThmb});
  \item $|I|\geq 3$, and $e_i\geq g$ for all $i\in [n]$ (see
    Figure~\ref{sf:mainThmc}).
  \end{enumerate}
\end{enumerate}
The fibers of $\phi$ over these trees are polyhedral complexes with
eight or sixteen  maximal cells of dimension at
most two. These complexes are closed under the action of $\ZZ_2$. Over
generic points, these fibers are pure complexes of dimension exactly
two. 
\end{theorem}
  \begin{figure}[ht]
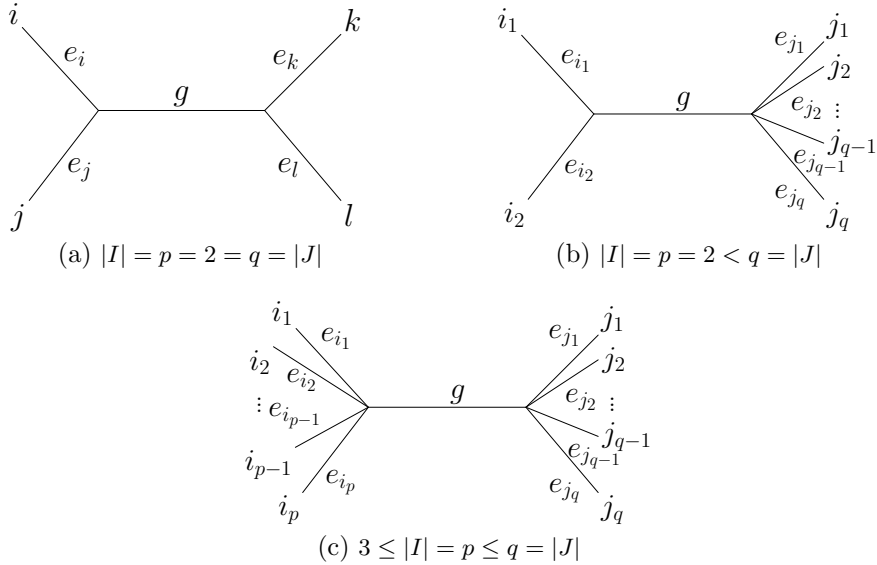

    \centering
    \subfloat[\mbox{$|I|=p=2 =q=|J|$}]{%
\includegraphics[scale=0.58]{angie.16}
\label{sf:mainThma}
}
\qquad \qquad 
 \subfloat[$|I|=p=2 <q=|J|$]{
\includegraphics[scale=0.55]{angie.15}
 \label{sf:mainThmb} 
}
\qquad
\subfloat[$3\leq |I|=p \leq q=|J|$]{
\includegraphics[scale=0.55]{angie.14}
\label{sf:mainThmc}
}
\caption{Tree topologies of mixtures of two star tree metrics.}
  \end{figure}

\section{Proof of Lemmas~\ref{lm:n=4},~\ref{lm:n=5}, and of
  Theorem~\ref{thm:mainResult} }\label{s:proofs}
Before presenting the proofs of the two main lemmas, we need an
algebraic characterization of the topology of three metrics that equal
$D\oplus \HH$ in terms of the entries of the star tree metrics $D$ and
$\HH$. Since the topology of star trees is highly symmetric we need
only treat the case of fibers of the mixture map over the quartet
$(12|34)$.  The general case follows by relabeling. 

Following the proof technique in~\cite{MatSt}, our description involves
relative differences among the entries of both metrics. We denote by
$a,d,c$ and $e$ the weights of the edges in $D$ adjacent to leaves 1
through 4. With the same convention, $\aaa, \dd,\cc$ and $\ee$ are used for the edges
in $\HH$. We label the middle edge of the quartet $(12|34)$ by $g$
(see Figure~\ref{sf:mainThma}).
Note that since the image of $\phi$ and the trees with a fixed
topology are cones, then our characterization must be 
invariant under multiplication by positive scalars. 

From the edge weights in the star tree metrics $D$ and $\HH$ we obtain
two $4\times 4$ nonnegative symmetric matrices.  To simplify notation,
we only show the coefficients of these matrices above the diagonal:
\[
D:=
\left(\begin{array}{cccc}
  0 & a+d & a+c & a+e \\
 & 0 & d+c & d+e \\
& & 0 & c+e\\
& & & 0 
\end{array}
\right)
\; ; \;
\HH:=\left(\begin{array}{cccc}
  0 & \aaa+\dd & \aaa+ \cc & \aaa+ \ee \\
 & 0 & \dd+ \cc & \dd+ \ee \\
& & 0 & \cc+\ee\\
& & & 0 
\end{array}
\right).
\]
We consider six new indeterminates $s,t,x,y,u$ and $w$, which  express
the relative differences between the matrices $D$ and $\HH$.

  \begin{minipage}{0.35\linewidth}
\begin{equation}\left\{
\begin{aligned}
  \aaa  + \cc & = a + c + s,\\
  \dd  + \ee & =  d +e + t,
\end{aligned}    \right.
\label{eq:3}\end{equation}
\end{minipage},
  \begin{minipage}{0.3\linewidth}
\begin{equation*}
\left\{
\begin{aligned}
  \aaa + \ee & =  a  +e +x,\\
  \dd  + \cc & =  d  + c + y,
\end{aligned}\right.
\end{equation*}
\end{minipage},
  \begin{minipage}{0.3\linewidth}
\begin{equation*}\left\{
\begin{aligned}
\aaa + \dd & =  a +d + u, \\
  \cc + \ee & =  c +e + w.
\end{aligned} \right.
\end{equation*}\end{minipage}
From \eqref{eq:3} we see that these new variables satisfy two
linear equations: 
\[
 s+t=x+y=u+w.
\]
Denote $s^+=\max\{s,0\}$ and similarly for the other five new
variables. Observe that $T_{13}=\max\{D_{13}, \HH_{13}\} =
D_{13}+s^+$, and similarly for the other variables. 
%
 The next result will be crucial in the proof of
Lemma~\ref{lm:n=4}.

\begin{proposition}\label{pr:charact(12|34)}
  With the above notation,  $T$ is the quartet
  $(12|34)$ \emph{if and only if}
$  u^+ + w^+ < s^+ + t^+ = x^+ + y^+$.
\end{proposition}
\begin{proof}
  Follow immediately from Theorem~\ref{thm:4ptCondition}. 
\end{proof}

\subsection{Proof of Lemma~\ref{lm:n=4}.}
  
Our strategy is elementary and it is based on a careful case by case
analysis. We let $T$ be the quartet $(12|34)$. We assume $T=D\oplus
\HH$, and we find necessary and sufficient conditions on the weights
of $T$ for this equality to hold. Our unique restrictions are given by
requiring all values of weights to be 
nonnegative. 
 We show that these conditions suffice to explicitly construct $D$ and
 $\HH$ from $T$. Proposition~\ref{pr:charact(12|34)} will be our main
 tool throughout the proof.

 Since the mixture map is symmetric on $D$ and $\HH$, without loss of
 generality we may assume $s\geq 0$. Thus, we need to analyze two
 cases according to the sign of $t$.

 If $t\geq 0$, we need $s+t=x+y$ and $s+t = x^+ + y^+$. Thus $x,y\geq
 0$. Likewise $s+t=u+w$ but $u+w \leq u^++ w^+ <s+t$, a
 contradiction. Hence $t<0$. By switching the roles of $s$ and $t$ we
 also conclude that $s$ and $t$ are both nonzero. Without loss of
 generality, we may assume $s>0$ and $t<0$.

 By the symmetry on $(12|34)$, we observe that \emph{both} pairs
 $(s,t), (x,y)$ have coordinates of \emph{opposite} sign.
  In addition, since $x+y=s+t<s$ and $s=x^++ y^+$, then we must have either
  $x=s,y=t$ or $x=t,y=s$. By analyzing  \eqref{eq:3} we conclude that either:
\[
s=x > 0,\, t=y<0 \; \Rightarrow\; 
 \aaa=a+ s-\frac{w}{2},\, \cc=c+\frac{w}{2},\, \dd=d+ u+\frac{w}{2}-s,\,
\ee=e+\frac{w}{2},
\]
or
\[
s=y> 0,\, t=x<0 \; \Rightarrow\; \aaa=a+\frac{u}{2},\,
\cc=c+ s-\frac{u}{2},\,\dd=d+\frac{u}{2},\,  \ee=e+ x-\frac{u}{2}.
\]
Note that these two cases are symmetric with respect to the order two
permutation that is the product of the transpositions $(1\, 2)$ and $(3\,
4)$.

We now consider the sign patterns of $u$ and $w$. Assume $u,w$ have
opposite sign. Then from $u+w=s+t$ we conclude $u^++w^+<s$ \emph{if
  and only if} ($0\leq u<s$ and $w=s+t-u<0$) \underline{or} ($0\leq
w<s$ and $u=s+t-w<0$). In case both $u$ and $w$ have the same sign,
we have no restrictions other than $u+w=s+t$.

Using Theorem~\ref{thm:4ptCondition}, we can compute the weights of the
edges of $T$ in terms of the variables $a,c,d,e,s,t,x^+,y^+, u^+$ and $w^+$, following the labeling
of Figure~\ref{sf:mainThma} 
with $i=1, j=2, k=3, l=4$. 
\small{
  \begin{equation}
\left\{
    \begin{aligned}
     & g = \frac{s-(u^++w^+)}{2}, e_1= a+\frac{x^++u^+}{2} , e_2=
      d+\frac{y^+-s+u^+}{2} , \\&e_3= c+\frac{y^++w^+}{2}, e_4=
      e+\frac{x^+-s+w^+}{2}.
    \end{aligned}
\right.\label{eq:edges}
\end{equation}
}
\normalsize

 \noindent Next, we compute each indeterminate in \eqref{eq:edges} according to our two sign patters for  $x$ and $y$. Case 1 corresponds to $s=x >
 0$ and $t=y<0$, whereas Case 2 gives $s=y > 0$ and $t=x<0$.  
We subdivide these two cases according to further sign constrains: $s>u\geq 0$
and $w\leq 0$ (Cases 1.1 and 2.1), $s>w\geq 0$ and $u\leq 0$ (Cases 1.2 and
2.2) or $u,w$ with the same sign (nonnegative in Cases 1.3 and 2.3 and
nonpositive in Cases 1.4 and 2.4).  In this way, we get expressions for
all our indeterminates and we obtain necessary properties they must
have.
\begin{remark} \label{rk:symmetriesCases} Notice that Cases 1 and 2
  are symmetric. More precisely, they are equivalent if we change the
  labels for leaves $1$ and $2$ by $3$ and $4$ respectively and we
  also change the variables $(a,d,u,w,x,y)$ by $(c,e,w,u,y,x)$, while
  leaving $s$ and $t$ fixed.  One needs to make these replacements
  between suitable pairs: 1.1 $\&$ 2.2, 1.2 $\&$ 2.1, 1.3 $\&$ 2.3 and
  1.4 $\&$ 2.4 respectively.   Therefore, it suffices to  study
  Cases 1.1 through 1.4. This correspondence will be crucial for
  the proof of Lemma~\ref{lm:n=5}. 
\end{remark}

\begin{itemize}

\item {\textbf{Case 1.1:} }
Suppose $s>u\geq 0$ and $w\leq 0$, then
\[
g  =  \frac{s-u}{2},\quad
e_1  =  a+\frac{s+u}{2} ,\quad
e_2  =  d-\frac{s-u}{2} ,\quad
e_3  =  c,\quad
e_4  =  e.
\]
Since $a\geq 0,u\geq 0$ the previous equations yield $e_1\geq
g$. Similarly, by requiring all variables $\aaa,\dd,\ee,\cc$ and $e_k$
to be nonnegative we
conclude that 
$e_1,e_2\geq g$, $d\geq s-u-w/2$ and
$c,e\geq -w/2$.

\item {\textbf{Case 1.2:}} Suppose $s>w\geq 0$ and $u\leq 0$, then
\[
g = \frac{s-w}{2}, \quad
e_1 = a+\frac{s}{2} , \quad
e_2 = d-\frac{s}{2}  , \quad
e_3 = c+ \frac{w}{2} , \quad
e_4 = e+ \frac{w}{2}. 
\]
In this case, from $a\geq 0$ and $w\geq 0$ we also have $e_1\geq
g$. Analogously, from $u\leq 0$ and $0 \leq w<s$ we obtain additional
necessary conditions: $d\geq s-u-w/2$ and $e_2=d-s/2\geq g-u\geq g$.

\item {\textbf{Case 1.3:}} Suppose $u,w\geq 0$, then
\[
g = \frac{s-(u+w)}{2},\;
e_1 = a+\frac{s+u}{2}  ,\;
e_2 = d-\frac{s-u}{2} ,\;
e_3 = c+ \frac{w}{2} ,\;
e_4 = e+ \frac{w}{2}.
\]
As before, we have $e_1\geq g$, and we obtain the extra conditions
$d\geq s-u-w/2$,
$e_2
\geq s-u-w/2-(s-u)/2=g$.

\item {\textbf{Case 1.4:}} Suppose $u,w \leq 0$, then
\[
g = \frac{s}{2} , \quad
e_1 = a+\frac{s}{2} , \quad
e_2 = d-\frac{s}{2}  , \quad
e_3 = c , \quad
e_4 = e.
\]
Here we have $e_1\geq g$,  $c,e \geq -w/2$ and 
$d\geq s-u-w/2$, thus
$e_2=d-s/2 \geq 
s/2-u-w/2 
\geq g$.
%
\end{itemize}

\medskip

From the previous analysis, we conclude that $s=x, t=y$ implies
$e_1,e_2\geq g$ and $d\geq s-u-w/2$.  Changing variables as stated in
Remark~\ref{rk:symmetriesCases}, we get analogous results for Cases
2.1 through 2.4. In particular, we obtain $e_3,e_4\geq g$ and $e\geq
s-w-u/2$. 
Hence, if $(12|34)$ is a mixture of two 
star tree
metrics we conclude $e_1,e_2\geq g$ or $e_3,e_4\geq g$.  \medskip

For the converse, we assume that $T$ is the quartet $(12|34)$ and that
it satisfies $e_1,e_2\geq g$ or $e_3,e_4\geq g$. Our goal is to
construct all possible pairs $D$, $\HH$ giving $T=D\oplus \HH$.  From
the ``if direction'' we know that such pairs $(D,\HH)$ belong to a
polyhedral complex of dimension at most 2. Each cell is parameterized
by the variables $u,w$. 

Assume that the condition $e_3,e_4\geq g$ \emph{does not} hold. The,  we
claim that Cases 1.1 through 1.4 are admissible \emph{if and only if}
$e_1,e_2\geq g$. We already know the ``only if direction'' from the
previous discussion. We now prove the converse: 


\begin{itemize}
\item \textbf{Case 1.1:} Assume $e_1,e_2\geq g$. We want to
  determine the values of $a,c,d,e$ and $\aaa,\cc,\dd,\ee$ so that
  $T=D\oplus \HH$.  From the previous analysis we easily obtain
  $c=e_3$, $d=e_2+g$ and $e=e_4$.  On the other hand, we have
  $s=u+2g$, $a=e_1-{(s+u)}/{2}=e_1-g-u$, $\aaa=e_1+g-w/2$,
  $\cc= e_3+w/2$, $\dd=e_2-g+w/2$ and $\ee=
  e_4+w/2$.  The only conditions on $u,w$ are determined by
  the nonnegativity of all weights in $D$ and $\HH$, namely, $u,w$ are
  free parameters satisfying $0\geq w \geq -2\min\{e_2-g, e_3,e_4\}$ and
  $e_1-g\geq u\geq 0$. The sign patterns for $s,t$ hold
  automatically by construction.
%
%
%
%
Hence, our necessary and sufficient conditions are $e_1,e_2\leq g$.

\item \textbf{Case 1.2:} As in the previous case, if we assume
  $e_1,e_2\leq g$ we get  $s=2g+w$,
 $a=e_1-g-w/2$,  $c=e_3-w/2$, $d=e_2+g+w/2$, $e=e_4-w/2$,
   $\aaa=e_1+g$ $\cc=e_3$,  $\dd=e_2-g+u$ and $\ee=e_4$, where $u$ and $w$
  are two free parameters satisfying $0\geq u\geq g-e_2$ and
  $2\min\{e_3,e_4,e_1-g\}\geq w\geq 0$.
%
%
%
%
%
Again, our necessary and sufficient conditions are
$e_1,e_2\geq g$ as well.

\item \textbf{Case 1.3:} Suppose $e_1,e_2\geq g$. From scratch we
  have $\cc=e_3$ and $\ee=e_4$ . We let $u,w\geq 0$ be two free
  parameters, so $s=2g+u+w$, $a=e_1-g-u-w/2$,
  $c=e_3-w/2$, $d=e_2+g+w/2$, $e=e_4-w/2$,
  $\aaa=e_1+g$ and $\dd=e_2-g$.
  In particular, the only restrictions over
  $u,w$ are $0\leq w\leq 2\min\{e_3,e_4,e_1-g\}$ and $0\leq
  u\leq e_1-g-w/2$.




\item {\textbf{Case 1.4:}} Assume $e_1,e_2\geq g$. Then $s=2g$,
  $a=e_1-g$, $c=e_3$, $d=e_2+g$ and $e=e_4$. Our free parameters are
  $u,w\leq 0$. We get $\aaa=e_1+g-w/2$, $\cc=e_3 +
  w/2$, $\dd=e_2-g+u+w/2$ and
  $\ee=e_4+w/2$. Therefore, $u,w$ must satisfy $0\geq w
  \geq-2\min\{e_3,e_4,e_2-g\}$ and $0\geq u\geq g-e_2-w/2$.
\end{itemize}

\medskip

  By the symmetries mentioned in Remark~\ref{rk:symmetriesCases}, we
  have similar results for Cases 2.1 through 2.4. The necessary and
  sufficient conditions are $e_3,e_4\geq g$ in case $e_1,e_2\geq g$
  does not hold. We also obtain $u, w$ as free parameters and their
  linear restrictions are translations of the restrictions from the
  corresponding previous four cases.

  Therefore, we have three possible scenarios for the
  fibers 
  to be nonempty:
\begin{enumerate}
\item $e_1,e_2 \geq g$ does \emph{not} hold. Then $e_3,e_4\geq g$ and
  the fiber of the mixture map is the union of the four families
  described by Cases 2.1 through 2.4, and their orbits under the
  action of $\ZZ_2$ (which interchanges $D$ and $\HH$). The fiber over
  such a point is a polyhedral complex with eight maximal cells (Cases
  2.1 through 2.4 and their $\ZZ_2$-orbits). Each cell
  is affinely isomorphic to a product of intervals, a trapezoid or a
  triangle in $\RR^2$.  If $e_1,e_2>g$, the fiber over this point is a
  pure complex of dimension two, i.e.\ the maximal cells have dimension two.
\item $e_3,e_4\geq g$ does \emph{not} hold. Then $e_1,e_2\geq g$ and the fiber
  of the mixture map is a polyhedral complex with eight maximal cells
 (Cases 1.1 through 1.4 and their $\ZZ_2$-orbits) of dimension at most
 two. 
\item $e_1,e_2,e_3,e_4\geq g$. Then, the fiber is a polyhedral complex
  with sixteen maximal cells defined by Cases 1.1--2.4 and
  their $\ZZ_2$-orbits. 
\end{enumerate}

As a reference for the proof of Lemma~\ref{lm:n=5}, we include the
description of all eight cases below. The free parameters are
$u,w$. Notice that in all cases, one of the star trees has at
least three of its edge weights independent of the values of $u$ and
$w$.
\begin{itemize}
\item \textbf{Case 1.1:} 
  $e_1-g\geq u\geq 0$ and $0 \geq w \geq -2\min\{e_2-g,e_3,e_4\}$,
  \[
  \left\{
    \begin{array}{llll}
      a  :=  e_1-g-u, &
      c  :=  e_3, &
   d :=  e_2+g, &
      e  :=  e_4, \\
      \bar{a}  :=  e_1+g-\frac{w}{2}, &
      \bar{c}  :=  e_3 +\frac{w}{2}, &
      \bar{d}  :=  e_2-g+\frac{w}{2}, &
      \bar{e}  :=  e_4 + \frac{w}{2}. 
    \end{array}
  \right.
  \]

\item \textbf{Case 1.2:} 
  $0\geq  u\geq g-e_2$ and $2\min\{e_1-g, e_3,e_4\}\geq w\geq 0$,
  \[
  \left\{
    \begin{array}{llll}
      a  :=  e_1-g-\frac{w}{2}, &
      c  :=  e_3-\frac{w}{2}, &
      d  :=  e_2+g+\frac{w}{2}, &
      e  :=  e_4-\frac{w}{2}, \\
      \bar{a}  :=  e_1+g, &
      \bar{c}  :=  e_3, &
      \bar{d}  :=  e_2-g+u, &
      \bar{e}  :=  e_4. 
    \end{array}
  \right.
  \]

\item {\textbf{Case 1.3:}} 
  $e_1-g-{w}/{2}\geq u\geq 0$ and $2\min\{e_1-g,e_3,e_4\}\geq w\geq 0$,
  \[
  \left\{
    \begin{array}{llll}
      a  :=  e_1-g-u-\frac{w}{2}, &
      c  :=  e_3-\frac{w}{2}, &
      d  :=  e_2+g+\frac{w}{2}, &
      e  :=  e_4-\frac{w}{2}, \\
      \bar{a}  :=  e_1+g, &
      \bar{c}  :=  e_3, &
      \bar{d}  :=  e_2-g, &
      \bar{e}  :=  e_4.
    \end{array}
  \right.
  \]

\item \textbf{Case 1.4:} 
  $0\geq u\geq g-e_2-{w}/{2}$ and $0\geq w\geq -2\min\{e_2-g,e_3,e_4\}$,
  \[
  \left\{\hspace{-0.2cm}
    \begin{array}{llll}
      a  :=  e_1-g, &\hspace{-0.1cm}
      c  :=  e_3, &\hspace{-0.1cm}
      d  :=  e_2+g, &\hspace{-0.1cm}
      e  :=  e_4, \\
      \bar{a}  :=  e_1+g-\frac{w}{2}, &\hspace{-0.1cm}
      \bar{c}  :=  e_3 + \frac{w}{2}, &\hspace{-0.1cm}
      \bar{d}  :=  e_2-g+ u+ \frac{w}{2} , &\hspace{-0.1cm}
      \bar{e}  :=  e_4 + \frac{w}{2}. 
    \end{array}\hspace{-0.35cm}
  \right.
  \]

\item \textbf{Case 2.1:} $ 2\min\{e_1,e_2,e_3-g\}\geq u\geq 0$ and
  $0\geq w\geq g-e_4$,
  \[
  \left\{
    \begin{array}{llll}
      a   :=  e_1-\frac{u}{2}, &
      c  :=  e_3-g-\frac{u}{2}, &
      d   :=  e_2-\frac{u}{2}, &
      e  :=  e_4+g+\frac{u}{2}, \\
      \bar{a}  :=  e_1, &
      \bar{c}  :=  e_3+g, &
      \bar{d}  :=  e_2, &
      \bar{e}  :=  e_4-g+w.
    \end{array}
  \right.
  \]

\item \textbf{Case 2.2:} 
  $0\geq u\geq -2\min\{e_1,e_2,e_4-g\}$ and $e_3-g\geq w\geq 0$,
  \[
  \left\{
    \begin{array}{llll}
      a  :=  e_1, &
      c  :=  e_3-g-w, &
      d  :=  e_2, &
      e  :=  e_4+g, \\
      \bar{a}  :=  e_1+\frac{u}{2}, &
      \bar{c}  :=  e_3 +g-\frac{u}{2}, &
      \bar{d}  :=  e_2+\frac{u}{2}, &
      \bar{e}  :=  e_4 -g+\frac{u}{2}. 
    \end{array}
  \right.
  \]

\item \textbf{Case 2.3:} 
  $2\min\{e_1,e_2,e_3-g\}\geq u\geq 0$ and $e_3-g-{u}/{2}\geq w\geq 0$,
  \[
  \left\{\hspace{-0.2cm}
    \begin{array}{llll}
      a  :=  e_1-\frac{u}{2}, &\hspace{-0.1cm}
      c  :=  e_3-g-\frac{u}{2}-w, &\hspace{-0.1cm}
      d  :=  e_2-\frac{u}{2}, &\hspace{-0.1cm}
      e  :=  e_4+g+ \frac{u}{2}, \\
      \bar{a}  :=  e_1, &\hspace{-0.1cm}
      \bar{c}  :=  e_3+g, &\hspace{-0.1cm}
      \bar{d}  :=  e_2, &\hspace{-0.1cm}
      \bar{e}  :=  e_4-g.
    \end{array}\hspace{-0.35cm}
  \right.
  \]

\item \textbf{Case 2.4:} 
  $0\geq u\geq -2\min\{e_1,e_2,e_4-g\}$ and $0\geq w\geq (g-e_4)-u/2$,
  \[
  \left\{
    \begin{array}{llll}
      a  :=  e_1, &
      c  :=  e_3-g, &
      d  :=  e_2, &
      e  :=  e_4 +g, \\
      \bar{a}  :=  e_1+ \frac{u}{2}, &
      \bar{c}  :=  e_3 + g- \frac{u}{2}, &
      \bar{d}  :=  e_2 + \frac{u}{2}, &
      \bar{e}  :=  e_4 -g + \frac{u}{2}+w. 
    \end{array}
  \right.
  \]
\end{itemize}

\subsection{Proof of Lemma~\ref{lm:n=5}.}
As in the proof of Lemma~\ref{lm:n=4}, our goal is to study necessary
and sufficient conditions under which a point has a nonempty fiber,
and also to describe these fibers. By symmetry, we need only consider
the two cases illustrated in Figure~\ref{fig:n=5Trees}. We treat each
topology separately, analyzing the possible quartets $(12|34),
(12|35), (12|45)$ on these trees.
Since we have a complete characterization of mixtures of two star tree
metrics on four taxa, we construct our candidates for the mixtures on
five taxa given each quartet on the trees. To finish, we glue our
partial star tree metrics together (whenever possible) to build our
candidates $D,\HH$ on five taxa. We show this by an exhaustive case by
case analysis for each quartet.


\begin{figure}[ht]
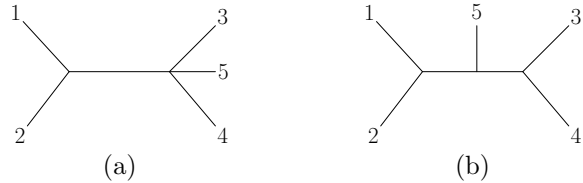

\centering
 \subfloat[]{
 \label{sf:ejemplob}
\includegraphics[scale=0.35]{angie.2}
 }
\qquad \qquad
\subfloat[]{
\label{sf:ejemploc}
\includegraphics[scale=0.35]{angie.3}
}
\caption{Non-star trees on five taxa, up to permutation of taxa.}
\label{fig:n=5Trees}
\end{figure}
\begin{proof}  [Proof for Figure~\ref{sf:ejemplob}]
  Let $T$ be as in Figure~\ref{sf:ejemplob}. As we said before, we
  have necessary and sufficient conditions for each quartet in $T$ to
  be a mixture of two star tree metrics. 
  Assume $T=D\oplus \HH$. Following the notation and proof-strategy of
  Lemma~\ref{lm:n=4}, we let $g$ be the weight of the unique internal
  edge in $T$ and $e_i$ the weight of the edge in $T$ adjacent to leaf
  $i$.  Similarly, we let $a,d,c,e,f$ (resp.\ $\aaa, \dd, \cc, \ee,
  \ff$) be the weights of the edges in $D$ (resp.\ $\HH$) adjacent to
  leaves 1 through 5, as illustrated on the right-most picture in
  Figure~\ref{fig:labelsEjemploc}. Since the  quartets $(12|34)$ and $(12|35)$
  involve leaves $1$ and $3$, we can always assume that
  $s>0$. However, for the quartet $(12|45)$ we need to consider
  \emph{both} signs for the difference
$(\aaa+\ff)-(a+f)$.
By considering the
restrictions to the three mentioned quartets we see that
the underlying tree  $T$ must satisfy the following conditions:
\[
e_1,e_2 \geq g \quad \text{ or } \quad e_3,e_4,e_5 \geq g.
\]

Assume $e_1,e_2\geq g$. Therefore, we know that at least we have to
analyze Cases 1.1 up to 1.4 for the three quartets, up to permuting $D$
and $\HH$. 
We also treat Cases 2.1 through 2.4 if $e_3,e_4\geq g$.

Suppose $(12|34)$ follows Case 1.1. Then, our  free parameters
$u,w$ satisfy $e_1-g \geq u \geq 0 $ and $0\geq w\geq
-2\min\{e_2-g,e_3,e_5\}$. 
%
The quartet $(12|35)$ has the same value of $s=\aaa+\cc-a-c>0$,
and $u$, so we do not need to switch the order of $D$ and $\HH$. We
claim that the fiber over $(12|35)$ can only correspond to Case
1.1. Assuming Cases 2.1 or 2.3, using that $g>0$ we see that
$\aaa=e_1+g-w/2> e_1=\aaa$, a contradiction. Likewise, for Cases 2.2
and 2.4 we obtain $\aaa=e_1+g-w/2 > e_1+u/2 =\aaa$, also a
contradiction.

Next, we claim that any point in the fiber over the quartet $(12|35)$ that
satisfies Case $1.2, 1.3$ or $1.4$, must also satisfy Case $1.1$. Call
$v=\cc+\ff-c-f$.  For Case $1.2$ we obtain $u=w=v=0$ so the point in
the fiber satisfies Case 1.1. Similarly, for Case 1.3 we get $w=v=0$,
and the point also satisfies Case $1.1$. Finally, for Case $1.4$, we
obtain $w=v$ and $u=0$, and the same conclusion holds.  Thus by
joining the weight values obtained for both quartets using Case 1.1 we
get $v=w$ and compatible values for all weights in $D$ and $\HH$.

To analyze the quartet $(12|45)$ we need to compute the difference
$(\aaa+\ff)-(a+f)$, since its sign \emph{determines the order} of $D$
and $\HH$ in expression \eqref{eq:3}. We obtain
$(\aaa+\ff)-(a+f)=(e_1+e_5+g)- (e_1+e_5-g-u)=2g+u>0$. Therefore, we do
\emph{not} need to switch the order of $D$ and $\HH$ in our
calculation.  
As before, there is only one possibility for the quartet $(12|45)$:
Case 1.1 must hold. The proof is analogous to the one of the quartet
$(12|34)$ \emph{because} we know that the order of $D$ and $\HH$ is
preserved.

Therefore, if $e_1,e_2\geq g$ and we are in Case 1.1 for the quartet
$(12|34)$ we have that $T=D\oplus \HH$
where
\[
\left\{\hspace{-.2cm}
\begin{array}{lllll}
a  =  e_1-g-u,&
c  =  e_3,&
d  =  e_2+g,&
e  =  e_4,&
f  =  e_5,\\
\bar{a}  =  e_1+g-\frac{w}{2},&
\bar{c}  =  e_3 +\frac{w}{2},&
\bar{d}  =  e_2-g+\frac{w}{2},&
\bar{e}  =  e_4 + \frac{w}{2},&
\bar{f}  =  e_5 + \frac{w}{2},
\end{array}\hspace{-.3cm}
\right.
\]
where $u,w$ are free parameters such that $e_1-g \geq
u\geq 0$, $0\geq w\geq -2\min\{e_2-g,e_3,e_4,e_5\}$.


On the other hand, assume Case 1.2 holds for $(12|34)$. Keeping the
notation of Case 1.1, a similar argument shows that $v=w$, $0\geq
u\geq g-e_2$ and $2\min\{e_3,e_4,e_5,e_1-g\}\geq w \geq 0$ on the
quartets $(12|34)$ and $(12|35)$, with compatible values for all edge
weights.

As before, to analyze the quartet $(12|45)$ we first need to compare
$\aaa+\ff$ with $a+f$. We have $\aaa+\ff-(a+f)=2g+w>0$, and so we do
not need to permute $D$ and $\HH$. Therefore, the only possibility for
the quartet $(12|45)$ is Case 1.2 which is compatible with the weights
obtained before. Hence, $T=D\oplus \HH$, where
\[
\left\{\hspace{-.2cm}
  \begin{array}{lllll}
 a  =  e_1-g-\frac{w}{2},&
  d  =  e_2+g+\frac{w}{2},&
  c  =  e_3-\frac{w}{2},&
  e  =  e_4-\frac{w}{2},&
f  =  e_5-\frac{w}{2},\\
  \bar{a}  =  e_1+g,&
  \bar{d}  =  e_2-g+u,&
  \bar{c}  =  e_3,&
  \bar{e}  =  e_4,&
  \bar{f}  =  e_5,
  \end{array}\hspace{-.3cm}
\right.
\]
where $u,w$ are free parameters satisfying $0\geq u\geq g-e_2$ and
$2\min\{e_1-g,e_3,e_4,e_5\}\geq w\geq 0$.

Next, assume Case 1.3 holds for $(12|34)$. As before, the same case
must also hold for $(12|35)$ and  we get $w=v$,  where
$2\min\{e_1-g,e_3,e_4,e_5\}\geq w\geq 0$ and $e_1-g-w/2 \geq u\geq 0$
and the values of all edge weights are compatible.
In addition, $\aaa+\ff-a-f =2g +(u+w)=s>0$, so $D$ and $\HH$ do not
switch in the expression for the quartet $(12|45)$. Like in the previous
situation, Case 1.3 also holds for $(12|45)$ and we conclude
$T=D\oplus \HH$ with
\[
\left\{\hspace{-.2cm}
  \begin{array}{lllll}
  a  =  e_1-g-u-\frac{w}{2},&\hspace{-.15cm}
  d  =  e_2+g+\frac{w}{2},&\hspace{-.15cm}
  c  =  e_3-\frac{w}{2},&\hspace{-.15cm}
e  =  e_4-\frac{w}{2},&\hspace{-.15cm}
f  =  e_5-\frac{w}{2},\\
  \bar{a}  =  e_1+g,&\hspace{-.15cm}
  \bar{d}  =  e_2-g,&\hspace{-.15cm}
  \bar{c}  =  e_3,&\hspace{-.15cm}
  \bar{e}  =  e_4,&\hspace{-.15cm}
   \bar{f}  =  e_5,
  \end{array}\hspace{-.35cm}
\right.
\]
where $u,w$ are free parameters satisfying 
$2\min\{e_1-g,e_3,e_4,e_5\}\geq w\geq 0$ and $e_1-g-w/2\geq u\geq 0$.

To finish, assume $(12|34)$ satisfies Case 1.4. As before, we can
easily see that the same case also holds for $(12|35)$ and we obtain
compatible values for all weights.
In this case, $\aaa+\ff-(a+f)=2g>0$, and $(12|45)$ also satisfies Case 1.4. Therefore
$T=D\oplus \HH$ with
\[
\left\{\hspace{-.2cm}
  \begin{array}{lllll}
  a = e_1-g,&\hspace{-.15cm}
  d = e_2+g,&\hspace{-.15cm}
  c = e_3,&\hspace{-.15cm}
  e = e_4,&\hspace{-.15cm}
  f = e_5,\\
  \bar{a} = e_1+g-\frac{w}{2},&\hspace{-.15cm}
  \bar{d} = e_2-g+ u+ \frac{w}{2} ,&\hspace{-.15cm}
  \bar{c} = e_3 + \frac{w}{2},&\hspace{-.15cm}
  \bar{e} = e_4 + \frac{w}{2},&\hspace{-.15cm}
  \bar{f} = e_5 + \frac{w}{2},
  \end{array}\hspace{-.35cm}
\right.
\]
where $u,w$ are free parameters satisfying $0\geq w\geq
-2\min\{e_2-g,e_3,e_4,e_5\}$ and $0\geq u\geq g-e_2-w/2$.
\medskip

Next, we analyze Cases 2.1 through 2.4 for the quartet $(12|34)$. We
keep the notation from the previous discussion. By symmetry on the
leaves $4$ and $5$, we know that Cases 2.1 through 2.4 also hold for
the quartet $(12|35)$.


Next, we argue that the quartet $(12|34)$ cannot satisfy Cases 2.1
through 2.4. Assume the contrary. By looking at the weight $\aaa=e_1$
coming from the quartets $(12|34)$ and $(12|35)$, we see that all
Cases 2.1 through 2.4 can hold for $(12|35)$, but they all imply that
$\aaa+\ff-(a+f)\leq -2g<0$. Hence, for the expression of the quartet
$(12|45)$ we have to switch the roles of $D$ and $\HH$. We claim that
Cases 1.1 through 1.4 are discarded for this quartet. By contraction,
assume $(12|34)$ satisfies Cases 2.1 or 2.3. Then by applying Cases
1.1 through 1.4 to the quartet $(12|45)$ we conclude that
$\dd>e_2=\dd$, a contradiction. Similarly, if  Cases $2.2$ or $2.4$
hold for $(12|34)$ we get $d\geq e_2-g<e_2=d$, also a
contradiction. Finally, by looking at the expression of $e$ coming from
Cases 2.1 through 2.4 for the quartet $(12|45)$, we obtain $e\leq
e_4-g<e_4+g=e$, a contradiction. In conclusion, Cases $2.1$ through
$2.4$ cannot hold for quartet $(12|34)$.

From the previous analysis we conclude: if $e_1,e_2\geq g$, then the
fiber over $T$ consists of Cases 1.1 through 1.4 for all three
quartets. Moreover, they all satisfy the same case. Conversely, assume $T$ is the mixture of $D$ and $\HH$. If
$e_1,e_2\geq g$ does not hold, this would imply that Cases 2.1 through 2.4
are the only candidates for $(12|34)$. But from the construction this
cannot happen. Therefore, we conclude that the fiber of the mixture
map over a point coming from a tree metric on five taxa with a single
internal edge is nonempty if and only if $e_1,e_2\geq g$.
\end{proof}
\begin{proof}   [Proof for Figure~\ref{sf:ejemploc}]
  As in the previous discussion, we label the edges of $T$ and of the
  star trees as in Figure~\ref{fig:labelsEjemploc}.
\begin{figure}[ht]
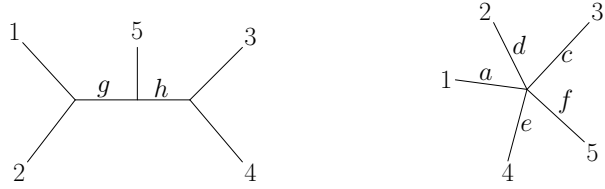

  \centering
  \includegraphics[scale=0.4]{angie.6} \qquad  \qquad \qquad
  \includegraphics[scale=0.4]{angie.4}
  \caption{Weights for edges in $T$ and in the star tree for $n=5$.}
\label{fig:labelsEjemploc}
\end{figure}
In what follows, we show that no tree metric with this topology is a mixture
of two star tree metrics.

Assume the contrary, and let $T$ be such a point. Then by
considering the quartet $(12|34)$ we know that $e_1,e_2\geq g+h$ or
$e_3,e_4 \geq g+h$. Without loss of generality we may assume $e_1,e_2
\geq g+h$. Now,
consider the quartet $(12|35)$. Since $e_1,e_2\geq g+h \geq g$, we have to
analyze all eight cases for $(12|34)$ and $(12|35)$. 

Suppose Cases 1.1 or 1.4 hold for the quartet $(12|34)$, so
$d=e_2+g+h$. If Cases 1.1 or 1.4 hold for $(12|35)$, we get
$e_2+g=d=e_2+g+h$, because the middle edge weight for $(12|35)$ is
$g$. Similarly, for Cases 1.2 and 1.3 we get $\aaa=e_1+g+h-w/2=e_1+g$
with $w\leq 0$ and $h>0$. Both situations lead to a
contradiction. Likewise, for Cases 2.1 and 2.3, we get
$d=e_2+g+h=e_2-{u}/{2}$ where $u\geq 0$, which cannot occur, whereas Cases
2.2 and 2.4 yield $d=e_2+g+h=e_2$, also a contradiction. Thus, Cases
1.1 and 1.4 are not admissible for the quartet $(12|34)$.


Next, suppose Cases 1.2 or 1.3 are satisfied by $(12|34)$. Then if we
assume Cases 1.1 or 1.4 for the quartet $(12|35)$ we get
$d=e_2+g=e_3+g+h+\frac{w}{2}$ with $w\geq 0$ and $h>0$, a
contradiction. Likewise, for Cases 1.2 and 1.3 we get
$\aaa=e_1+g+h=e_1+g$ which cannot happen.  For Cases 2.1 and 2.3 we get
$\aaa=e_1+g+h=e_1$, and Cases 2.2 and 2.4 yield 
$\aaa=e_1+g+h=e_1+{u}/{2}$ where $u\leq 0$, which cannot occur.



Finally, we analyze the Cases 2.1 through 2.4 for the quartet $(12|34)$. By
the previous discussion, we know that 
the quartet $(12|35)$ can only satisfy Cases 2.1 through 2.4 as
well. As in the previous discussion, we work with the weights $\cc$ and $\aaa$  and we  get
contradictions for all possible expressions of the quartet
$(12|35)$. This concludes our proof.
%
\end{proof}
\begin{remark}\label{rk:keyn=5}
  Note that in the proof of Figure~\ref{sf:ejemplob}, we found that
  all the quartets $(12|34)$, $(12|35)$ and $(12|45)$ satisfy the same
  case (either 1.1, 1.2, 1.3 or 1.4) and the fiber is a polyhedral
  complex whose maximal cells have dimension at most two. The free
  parameters involved on different quartets are \emph{the same} when
  restricting to each quartet, i.e. $w=v$ in the notation of the
  proof. This observation will be \emph{essential} to prove
  Theorem~\ref{thm:mainResult}.
\end{remark}

\subsection{Proof of Theorem~\ref{thm:mainResult}} 
  It remains to prove the item (ii) when $n\geq 6$. 
  From Lemma~\ref{lm:n=5}, it is clear that we have only one topology
  on $n$ taxa that gives a tree metric in the image of the mixture
  map, besides the star tree topology.  This topology is the one where
  the tree has a single internal edge, and it is illustrated in
  Figure~\ref{fig:nGralT'}. It is constructed by gluing two star trees
  with labels in the disjoint sets $I=\{i_1,\ldots, i_{p}\}$ and
  $J=\{j_1,\ldots, j_q\}$ with $I\sqcup J=[n]$. As before, we let $g$ be the
  weight of the unique internal edge and $e_k$ be the weight of the
  edge adjacent to the leaf $k\in [n]$.
\begin{figure}[ht]
  \centering
 \includegraphics[scale=0.5]{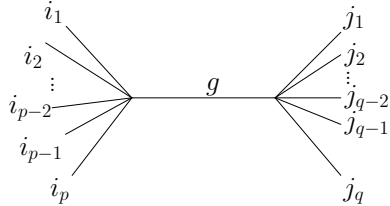}
\caption{Non-star candidate $T$ for a mixture of two star tree metrics on $n$ taxa.}
\label{fig:nGralT'}
\end{figure}
By symmetry on the taxa, we can assume $1,2 \in I$ and $3,4\in
J$. Moreover, we assume $|I|=p\leq q=|J|$. We claim that we have only two
cases to deal with, namely $|I|=2$ or $|I|>2$.

We first restrict the tree metric to the quartets $(12|34)$ and
$(12|3j)$, with $j>4$. Then, as we saw in the proof of
Lemma~\ref{lm:n=5} and in Remark~\ref{rk:keyn=5}, both quartets
satisfy the same case (among Cases 1.1 through 1.4) and $e_1,e_2\geq
g$. Likewise, the same case will hold for all quartets $(12|kl)$ where
$k,l\in J$.

Assume $|I|=2$. 
By gluing the previous partial solutions together we get that the
fiber over a tree metric with a single internal edge is a polyhedral
complex with eight maximal cells of dimension at most two if and only
if $e_1,e_2\geq g$. All other non-star tree metrics have empty fibers.

Now, we assume $|I|>2$. By switching the roles of $I$ and $J$
we conclude that the fiber is nonempty if and only if $e_i\geq g$ for
all $i\in [n]$ and that all the quartets $(kl|34)$ (resp.\ $(12|kl)$)
with $k,l\in I$ (resp.\ $k,l\in J$) satisfy the same case among Cases
2.1 through 2.4 (resp.\ Cases 1.1 through 1.4).  The correspondence
between Cases 1 and 2 mentioned in Remark~\ref{rk:symmetriesCases}
guarantees that we can construct $D$ and $\HH$ as in the case of five
taxa. Thus, the fiber over such tree metric is a polyhedral complex
with eight maximal cells of dimension at most two that is closed under
the natural action of $\ZZ_2$. Each one of these maximal cells is
isomorphic to a product of intervals, a trapezoid or a triangle in $\RR^2$.
%

\section{Tropical secant varieties and the space of star
  trees}\label{s:TropicalSecantVarieties}
In this section, we identify mixtures of finitely many star tree
metrics with relate tropical secant varieties of the cone of star tree
metrics.  For the basic definitions on tropical geometry we refer the
reader to \cite{RGStT}. For a gentle introduction to applications of
tropical geometry to phylogenetics, we recommend~\cite[Chapters
2-3]{ASCB}.

In~\cite[Theorem 3.4]{TropGrass}, Speyer and Sturmfels showed that the
space of (possibly negatively) weighted trees on $n$ taxa is (up to
sign) the tropicalization of the classical Grassmannian
$\mathscr{G}(2,n)$, called the \emph{tropical Grassmannian}. In
particular, the space of star trees is the tropicalization of the
toric variety parameterized by the monomial map:
\begin{equation}\label{eq:monomialMap}
  \psi\colon \CC[t_1^{\pm 1}, \ldots, t_n^{\pm 1} ] \to \CC[x_{ij}^{\pm 1}: 1\leq
  i<j\leq n]  \quad \psi_{ij}(\underline{t})=\alpha_{ij}t_it_j,
\end{equation}
for fixed scalars $\alpha_{ij}\in (\CC^*)^{\binom{n}{2}}$.
This toric variety is cut out by the ideal
  \begin{align*}
    I&=\langle x_{ij}-\alpha_{ij}t_it_j: 1\leq i< j\leq n \rangle
    \bigcap \CC[x_{ij}^{\pm 1}: 1\leq i< j\leq n ]  \\&=\langle
    x_{ij}x_{kl}-\frac{\alpha_{ik}\alpha_{kl}}{\alpha_{ik}\alpha_{jl}}
    x_{ik}x_{jl}\;;\;
    x_{il}x_{jk}-\frac{\alpha_{ik}\alpha_{kl}}{\alpha_{ik}\alpha_{jl}}
    x_{il}x_{jk} : i,j,k,l\rangle,\label{eq:5}
  \end{align*}
  where we identify $x_{ij}$ with $x_{ji}$ and $\alpha_{ji}$ with
  $\alpha_{ij}$ if $j<i$, and similarly for the other variables.  The
  $n$-dimensional linear space $L$ spanned by the star trees is the
  tropical variety of tropical rank one symmetric matrices. The space
  of star tree metrics is the intersection of $L$ with the positive
  orthant. 


  From expression~\eqref{eq:monomialMap} we know that $L$ is generated
  by the set of all vertices of the second hypersimplex
  $\Delta(2,n)\subset \RR^n$, i.e.~the exponents $\{e_i+e_j: 1\leq
  i<j\leq n\}$ of our monomial map $\psi$. The polytope has
  $\binom{n}{2}$ vertices and dimension $n-1$. We refer to
  \cite[Chapter~9]{GBandCP}, \cite[Chapter~6 $\S$3]{GKZ} for the historical
  background and basic properties of this well-studied
  polytope. 


As we mentioned in Section~\ref{s:Intro}, the mixtures of $k$ star tree
metrics are the points in the $k$-tropical secant of the space of star tree metrics. 
Due to the positivity restriction on the edge weights, \emph{tropical
  secants of cones} define the appropriate setting to analyze mixtures
of finitely many star tree metrics, as we now show. Given a cone
$\mathscr{L}$ in $\RR^m$, we define its $k$-tropical secant as the
collection of points $x$ in $\RR^m$ that can be written as
\[x_i=\big(p^{(1)} \oplus \ldots \oplus
p^{(k)}\big)_i:=\max\{p^{(1)}_i, \ldots, p^{(k)}_i\} \qquad \forall \;i\in
[n],
\] for suitable points $p^{(1)}, \ldots, p^{(k)}$ in
$\mathscr{L}$.  
Following the notation of \cite[Theorem 2.1]{tropicalDevelin}, we
replace the linear space $L$ by a rational polyhedral cone. Likewise,
the generators of $L$ are replaced by the extremal rays of this
cone. Our approach follows the \emph{max} convention for the tropical
semiring, as opposed to the $\emph{min}$ convention
of~\cite{tropicalDevelin}.
%
%
%
 Among the many results of~\cite{tropicalDevelin} that can be extended 
 to the cone setting, we have:
 \begin{theorem} \label{thm:2.1ForCones} 
   Let $\mathscr{L}\subset \RR^{m}$ be a polyhedral \emph{cone} with
   $d$ extremal rays. Consider the associated matrix $M_{\mathscr{L}}\in
   \RR^{d\times m}$ whose rows equal the previous $d$ rays. Let
   $V_{\mathscr{L}}=\{v_1, \ldots, v_m\}$ be the $m$-point configuration of
   $\RR^d$ given by the columns of $M_{\mathscr{L}}$. Then a vector $x\in \RR^m$
   is in the $k$-th tropical secant of ${\mathscr{L}}$ (with the $\max$
   convention) \emph{if and only if} the lower envelope of the
   polytope formed by the height vector $x$ has $k+1$ facets given by
   linear functionals with only \emph{positive} coefficients whose
   union contains each point of $V_{\mathscr{L}}$.
\end{theorem}
Notice that points in the various tropical secants of $\mathscr{L}$
define special regular subdivisions of the polytope generated by
$V_{\mathscr{L}}$.  The positivity condition for the coefficients
comes from the fact that we work with cones rather than linear spaces.
For simplicity, we call such regular subdivisions \emph{positive
  regular subdivisions}. Other than these small differences, the proof
follows line by line the proof of \cite[Theorem~2.1]{tropicalDevelin}.
The cone $\mathscr{L}$ of star tree metrics has $n$ extremal rays,
namely, the \emph{cut metrics}
corresponding to the partition $\{i\} \bigsqcup ([n]\smallsetminus
\{i\})$, as we vary $i\in [n]$. These metrics are defined as
$D_i(i,j)=1$ for $j\neq i$ and $0$ otherwise. The integer matrix
$M_{\mathscr{L}}$ is the matrix of exponents of the monomial map
$\psi$ from~(\ref{eq:monomialMap}).




\medskip

Another interesting question to study is the relationship between the
mixtures of star tree metrics, the tropical secants of cones and the
secondary fan of $\Delta(2,n)$.  By \cite[Theorem~2.1,
Corollary~2.2]{tropicalDevelin} we know that each regular subdivision
in the $k$-th tropical secant variety of the space of star trees
corresponds to a polyhedral cone of height vectors. In particular,
this yields a decomposition of this tropical secant variety as a
polytopal complex. In the case of the $k$-th tropical secant variety
of a pointed cone we have:
\begin{corollary}
  The $k$-th tropical secant variety of a pointed cone $\mathscr{L}$
  is a cone
  from $\mathscr{L}$ over a polytopal complex, which we call the $k$-th tropical
  secant complex of $\mathscr{L}$. The faces correspond to positive regular
  subdivisions of $V_{\mathscr{L}}$ in which there exist $k+1$ facets containing
  all of the points, where a face $F$ contains a face $G$ if the
  regular subdivision associated to $F$ refines the one of 
$G$.
\end{corollary}
Consider the spaces of star trees, of star tree metrics, and their
corresponding tropical secant complexes. We know that the complex
corresponding to star tree metrics 
is strictly contained in the complex corresponding to star
trees. 
In particular, it would be desirable to study how the cell structures
of both complexes relate. In addition, one can investigate how the
cell structure of this complex restricted to the space of tree metrics
compares to the subdivision of the tropical Grassmannian
$\mathscr{G}(2,n)$ into topological types.

In \cite{DM}, Cartwright and Chan study the possible \emph{star tree ranks}
of dissimilarity maps (see
Definition~\ref{def:dissMap}). 
These ranks are defined as the minimum $r$ such that any dissimilarity
map 
can be written as a mixture of $r$ star trees. If no such $r$ exists,
we say that the matrix has infinite star tree rank. We define the
\emph{star tree metric rank} of a dissimilarity map in an analogous
way. It follows from the definition that the star tree metric ranks
are bounded below by the star tree ranks.

In~\cite[Theorem 5]{DM} the authors show that for $n\geq 3$, any
dissimilarity map on $n$ taxa can be written as the mixture of $n-2$
star trees, possibly with negative weights on their edges. Moreover,
this bound is tight. 
By comparing their result with Lemma~\ref{lm:n=4}, we see that any
metric on four taxa is a tropical mixture of two star trees, but not
necessarily two star tree {metrics}.

We now rephrase these properties in the language of tropical geometry. Since the configuration of points associated to $\Delta(2,n)$ is in
convex position, \cite[Corollary~2.3]{tropicalDevelin} shows that for
$k\gg 0$, the $k$-th tropical secant (called the \emph{$\infty$-th
  tropical secant}) of the space of star trees equals the ambient
space $\RR^{\binom{n}{2}}$.  Notice that \cite[Theorem 5]{DM} gives us
$n-2$ as a sharp lower bound on such $k$'s. 
%
%
However, the space of tree metrics does not equal the $\infty$-th
tropical secant variety of the star tree metrics. For example,
Lemma~\ref{lm:n=4} implies that the cut metric $D$ defined by the
quartet $(12|34)$, where $D(1,3)=D(1,4) =D(2,3)=D(2,4)=1$ and $0$
otherwise, has star tree metric rank bigger than two. Furthermore,
this rank is infinite. 
By contradiction, assume that for some $N\geq 1$ we can write
$D=\bigoplus_{k=1}^N S_k$, where all $S_k$ are star tree metrics. For
each $k,i$, we let
$e_i^{(k)}$ be the weight of the edge pendant to leaf $i$ in $S_k$. Then
we have $S_k(1,2)=e_1^{(k)}+e_2^{(k)}\leq D(1,2)=0$ and
$S_k(3,4)=e_3^{(k)}+e_4^{(k)}\leq D(3,4)=0$. Since all weights are nonnegative, we conclude $e_{i}^{(k)}=0$ for all
$i,k$, hence $D=0$, a contradiction. This example can be extended to
higher number of taxa, by considering any \emph{cut metric} given by a
partition of $[n]$ where, at least, two subsets have more than one
element.  In particular, these results together with
\cite[Corollary~2.3]{tropicalDevelin} imply that regular subdivisions
of the second hypersimplex need not be positive, as one would expect
from the constructions.

As we see from the previous discussion, mixtures of star trees and of
star tree metrics have completely different behavior, although the
underlying combinatorics in both cases are closely related. Mixtures
of star trees were extensively studied in~\cite{DM}, but basic
questions on the mixtures of star tree metrics remain open. Among
these are the characterization of $k$-mixtures of star tree metrics
for $k\geq 3$, 
and the lack of effective membership tests for these sets. We believe that a serious study
of tropical secants of cones will help us to attack most of these open
problems.


\section*{Acknowledgments}\label{sec:aknowledgements}
This project developed from a course project in Lior Pachter's
Math~239 graduate course at UC Berkeley.  I wish to thank Lior Pachter
and Bernd Sturmfels for their guidance and suggestions, and also
Dustin Cartwright for useful discussions. Finally, I thank the
anonymous referee for careful reading and many clarifying suggestions.

\bibliography{bibliography}{}
\bibliographystyle{abbrv}
\bigskip

\end{document}